\documentclass[11pt, reqno, english]{amsart}
\usepackage{amsthm,amsmath,amssymb}
\usepackage{graphicx} 
\usepackage{xcolor} 
\usepackage{charter}
\usepackage{euler}
\usepackage[T1]{fontenc}
\usepackage[left=1in,right=1in,top=1in,bottom=1in]{geometry}

\newcommand{\e}{\varepsilon}
\newcommand{\f}{\varphi}

\newcommand{\im}{\textrm{im}}

\newcommand{\RR}{\mathbb{R}}
\newcommand{\R}{\mathcal{R}}

\newcommand{\Cech}{\operatorname{\check{C}ech}}

\newcommand{\rch}{\operatorname{Rch}}

\theoremstyle{plain}
\newtheorem{theorem}{Theorem}
\newtheorem{lemma}[theorem]{Lemma}
\newtheorem{corollary}[theorem]{Corollary}
\newtheorem{proposition}[theorem]{Proposition}

\theoremstyle{definition}
\newtheorem{definition}[theorem]{Definition}

\newtheorem{remark}[theorem]{Remark}

\theoremstyle{remark}

\title{Detecting invariant manifolds of dynamical systems using persistent homology}
\author{Stavros Anastassiou}
\address{Department of Mathematics, University of Western Macedonia, Kastoria, Greece}
\email{sanastassiou@gmail.com}
\author{Žiga Virk$^{1}$}
\address{Faculty of Computer and Information Science, University of Ljubljana, and institute IMFM, Ljubljana, Slovenia}
\email{ziga.virk@fri.uni-lj.si}

\begin{document}
\maketitle
\footnote{Corresponding author.}
\begin{abstract}
    We use methods of Persistent Homology Theory to study invariant manifolds of dynamical systems. We first establish connections between the persistence diagrams of two sets which are close to each other, with respect to the Hausdorff distance. We then apply these results to study properties of limit sets of specific dynamical systems, by using the persistence diagram of a numerically obtained sample set. Under mild assumptions, we show how to use numerical data to state analytical results concerning the geometry of the limit sets.
\end{abstract}
\section{Introduction}
The identification and subsequent geometric characterization of invariant manifolds are of central importance in the study of dynamical systems, since these manifolds act as the phase space ``skeleton" and can be used to characterize the asymptotic behavior of the orbits.

This is why a number of methods, both analytical and numerical, have been developed to attack this problem; the interested reader can find various perspectives in \cite{Haro et al, Haro-Llave, Nipp-Stoffer, Krauskopf, Kuznetsov} and the references therein. 

Topological Data Analysis arose from the wider context of computational topology \cite{Mis}. It provides a rigorous mathematical framework for studying the ``shape" of data sets \cite{Edels, Carlsson}. In the context of dynamical systems, these data sets typically manifest as numerically generated periodic orbits, tori, or strange attractors within the phase space. By constructing a filtration of simplicial complexes from such data, Persistent Homology (PH) (see \cite{Virk}) effectively distinguishes short-lived topological artifacts from persistent, robust topological features that signify true underlying geometric structures.

This suggests that, by analyzing the persistent topological signatures of densely sampled trajectories of dynamical systems, we can locate and study geometric structures present in the phase space of a system and this is the approach we pursue in this article (regarding the problem of analyzing the homological properties of a general manifold based on a sample of it, see also \cite{Smale}). \textbf{As the main contribution of this paper we present a framework for applying PH to detect and analyze invariant manifolds}. 

The utility of PH in analyzing dynamical systems has been demonstrated in recent studies. In \cite{Kramar}, they use PH to characterize flow patterns in data series extracted from large complex systems, in \cite{Bohlsen} they distinguish topologically different orbits of a magnetic field with the use of PH, while in \cite{Perea} they use embedding techniques to describe the PH of quasiperiodic data. This work aims to extend the applicability of PH to dynamical systems by establishing a theoretical framework that enables rigorous analytical assertions regarding the existence and geometry of limit sets.

The remainder of this paper is structured as follows. Section 2 contains the definitions that will be used this work. In section 3 we develop the main theoretical framework of this work. Based on the Stability Theorem of PH, we establish correspondences between the persistence diagrams of two sets that are close with respect to the Hausdorff distance and present results that allow for the rigorous extraction of the homological properties of a set directly from its persistence diagram. Finally, in section 4, we apply these results to specific dynamical systems. We manage to establish the existence of periodic orbits in two planar vector fields and of an invariant torus in a 3d system. We also investigate the geometry of the limit set present in the phase space of a system in dimension 4. Although we consider only vector fields, our approach can be utilized in the case of diffeomorphisms as well. 

\section{Preliminaries}
A \textbf{dynamical system} on a metric space $(X, d)$ is described by a continuous map $f: X \to X$ or a flow $\phi: \mathbb{R} \times X \to X$.

Let $x\in X$. In the case of maps, we define:
\begin{itemize}
\item {the positive orbit as $O^+(x) := \{f^n(x) : n \in \mathbb{N}\}$}
\item {an orbit segment of length $N\in \mathbb{N}$ as $O_N(x) := \{f^n(x) : 0 \le n \le N\}$ or $O_{[n,m]}:= \{f^k(x) : n \le k \le m\}$, for the orbit segment in the ``time interval" $[n,m]$}
\item {the tail of the orbit as $O_{\ge N}(x) := \{f^n(x) : n \ge N\}$.}
\end{itemize}
while the analogous definitions for flows are:
\begin{itemize}
\item {$O^+(x) := \{\phi(t,x) : t\geq 0\}$}
\item {$O_T(x) := \{\phi(t, x) : 0 \le t \le T\}$, $O_{[t_1,t_2]}(x) := \{\phi(t, x) : t_1 \le t \le t_2\}$}
\item {$O_{\ge T}(x) := \{\phi(t, x) : t \ge T\}$.}
\end{itemize}

The $\omega$-\textbf{limit} set $\omega(x)$ of a point $x$ is the set of all limit points of its positive orbit $O^+(x)$. It represents the long-term, asymptotic behaviour of the system starting at $x$.
Formally:
$$\omega(x) = \bigcap_{N \ge 0} \overline{O_{\ge N}(x)}.$$

In this article, we want to study $\omega$-limit sets using finite segments of orbits. To measure the ``proximity" of these segments to the limit sets of interest, we use the \textbf{Hausdorff distance}, which, for two non-empty closed and bounded subsets $A$ and $B$ of $X$, is defined as:
$$d_H(A, B) = \max \left\{ \sup_{a \in A} d(a, B), \sup_{b \in B} d(b, A) \right\},$$
where $d(a, B) = \inf_{b \in B} d(a, b)$ is the distance from point $a$ to the set $B$.

If the positive orbit $O^+(x)$ of a point $x$ is relatively compact, then $\lim_{N \to \infty} d_H(\overline{O_{\ge N}(x)}, \omega(x)) = 0$ (\cite{Jurga - Todd}). We would like to have an analogous property for a finite segment of the orbit. For this, we need the following definition:
\begin{definition}
The $\varepsilon$-covering time for the set $A\subseteq X$ is defined as follows:
\[
T(\varepsilon, A) = \inf \left\{ n\in \mathbb{N} \ \middle| \ \forall z \in A, \exists \ m \in \{0,..,n\}:\ d(f^m(x), z) < \varepsilon \right\},
\]
in the case of maps and as:
\[
T(\varepsilon, A) = \inf \left\{ T > 0 \ \middle| \ \forall z \in A, \exists t \in [0, T]\ :\  d(\phi(t, x), z) < \varepsilon \right\},
\]
in the case of flows.
\end{definition}

\begin{remark}
This definition is, of course,  equivalent to:
$$T(\varepsilon, A) = \inf \left\{ T > 0 \ \middle| \ \sup _{a\in A}d(a, \gamma_{[0, T]}(x)) < \varepsilon \right\}.$$
It represents the minimal time needed for a set A to be $\epsilon$-close to an initial segment of the orbit. 
\end{remark}

According to \cite{Jurga - Todd}, if $\omega (x)$ is non-empty and compact, for some point $x$, $T(\epsilon,\omega(x))$ exists and it is a positive real number, while if the orbit $O^+(x)$ is relatively compact then, $\forall \epsilon >0,\ \exists n,m \in \mathbb{N}:\ d_H(O_{[n,m]},\omega(x))<\epsilon$. Thus, one can find a finite orbit segment which is arbitrarily close, in the sense of Hausdorff distance, to the limit set of this orbit. In what follows, we shall assume that the orbit segment we are studying is arbitrarily close to the sought-after limit set.

Another assumption we shall make concerns the reach of a closed subset.
\begin{definition}Let $S$ be a closed subset of $X$.
    \begin{itemize}\item[]
        \item The \textbf{medial axis} $ax(S)$ of $S$ is the set of points in $X$ that do not have a unique closest point on $S$.
    \item The \textbf{reach} of $S$ is $\rch(S)=\inf_{y\in ax(S)}\min_{x\in S} d(x,y)$.
    \end{itemize}
\end{definition}
In what follows, $S$ will usually be a set, consisting of points sampling the orbit segment we are interested in. The open $r$ ball around $x$ will be denoted as $B(x,r)=\{y\in X \mid d(x,y)<r\}$. From this set, we shall construct the (open) \textbf{Cech complex} $\Cech(S, r)$ as follows:
\begin{itemize}
    \item The vertices set is $S$ itself.
    \item Cech complex $\Cech(S, r)$ consists of simplices $\sigma \subseteq S$, for which $\cap_{x \in \sigma} B(x, r) \neq \emptyset$.
\end{itemize}
In this way, a filtration is defined, consisting of all complexes at all positive scales, that is, the Cech complexes, along with the natural inclusions $\iota_{r_1, r_2}\colon \Cech(S, r_1) \hookrightarrow \Cech(S, r_2)$ for $r_1 \leq r_2$ form the Cech filtration $\{\Cech(S, r)\}_{r >0}$. Observe that when $X$ is an Euclidean space, the nerve theorem implies that $\Cech(S, r)$ is homotopy equivalent to the open $r$-neighborhood of $S$. 

\textbf{Persistent homology} is obtained by applying the homology functor with coefficients in a field (in any dimension $q$) to this filtration, thus obtaining a collection of vector spaces (say, $\{H_q(\Cech(S, r))\}_{r >0}$) with the linear bonding maps. It turns out that in the relevant cases, the obtained structure decomposes as a direct sum of interval modules and can be represented as a \textbf{persistence diagram}: a point with coordinates $(r_1,r_2)$ is contained in this diagram if an interval module with the initial  and final ``times" $r_1,\ r_2$ exists in this decomposition (see \cite{Virk} for details). Thus, the persistence diagram of $S$ (denoted by $PD(S)$ in what follows) contains information regarding the homology of $S$.

If two sets are close to each other, it is only natural to expect that their corresponding persistence diagrams are also close to each other. We shall see in the next section that this is true, as long as we measure the proximity of two persistence diagrams with the bottleneck distance.

\begin{definition}
    Let $D$ and $D'$ be two persistence diagrams, viewed as multisets of points in the extended upper half-plane 
$\Omega = \{ (b, d) \in \mathbb{R} \times (\mathbb{R} \cup \{\infty\}) \mid b < d \}$, 
along with the diagonal $\Delta = \{(t, t) \mid t \in \mathbb{R}\}$ representing the trivial interval. Here we consider $\Delta$ as single point. For formal reasons we assume each persistence diagram contains infinitely many copies of the trivial interval $\Delta$.

The \textbf{bottleneck distance} $d_B(D, D')$ is defined as:
\begin{equation}
    d_B(D, D') = \inf_{\gamma \colon D \to D'} \sup_{x \in D} \|x - \gamma(x)\|_\infty
\end{equation}
where $\gamma: D \to D'$ ranges over all bijections between $D$ and $D'$. Observe that when persistence diagrams are finite in the sense that they contain only finitely many non-diagonal points, the bijections $\gamma$ will map copies of $\Delta$ mostly to itself. Furthermore, since $\Delta$ can be represented by any point of the form $(t,t)$, we declare $\| (a,b)-\Delta\|_\infty$ to equal $\min_{t\in \RR} \| (a,b)-(t,t)\|_\infty= \| (a,b)-((b+a)/2,(b+a)/2)\|_\infty=(b-a)/2$.

\end{definition}

We wish to use the information contained in the persistence diagram of our sample set $S$ to prove results concerning the geometry of the $\omega$-limit set we are interested in, provided that they are close to each other. To achieve this, in the next section, we use the Stability Theorem to state general results towards this direction. 
\section{Persistence diagrams: stability and reconstruction results}

In this section we present several results explaining how PH of a sample set $S$ can reveal the information on homology and persistent homology of a set $X$.
They include a Stability result (the fact that PH is continuous), a Reconstruction result (deducing the homotopy type of a space from a sample), and an inference result (deducing the homology of a space from a sample). There are many versions of such results in the literature pertaining PH. In this section we provide the selection of such results that are used later in the paper.

Of central importance is the following:

\begin{theorem}
\label{ThmStab}
[Stability Theorem](\cite{Virk}) 
Let $S, S' \subseteq X$ be compact subspaces of a metric space $X$. If $D, D'$ are the corresponding persistence diagrams in any dimension obtained via the open Cech filtration with the ambient space $X$, then $d_B(D, D') \leq d_H(S, S').$
\end{theorem}

Thus, indeed, two sets which are close with respect to the Hausdorff distance have persistence diagrams which are close, with respect to the bottleneck distance. This is essentially a continuity statement, stating that the map $X \mapsto PD(X)$ is $1$-Lipschitz in the relevant metrics. We can therefore draw conclusions regarding the geometry of a set by studying the PD of a nearby set. 

The next proposition illustrates this.

\begin{proposition}
\label{CorStab}
Assume $S$ and $S'$ are compact subsets of a metric space $X$ with $d_H(S,S')\leq \e$. Then the following hold for the persistence diagrams obtained via the open Cech filtrations in $X$.
\begin{enumerate}
    \item The bottleneck distance between the geometric descriptors $PD(S)$ and $PD(S')$ is at most $\e$.
    
    \item Let $t<\infty$ be the largest finite death time of a 0-dim bar of $PD(S)$. Then $S'$ cannot be represented as the disjoint union of two subspaces that are more than $2t + 2\e$ apart. 
    
    \item Let $t<\infty$ be the largest finite death time of a 0-dim bar of $PD(S)$. If $S'$ is path-connected, then $d_H(S,S')\geq t$.

    \item Assume $S$ is discrete and let $t$ be the largest distance between the points of $S$. If $\e< t/2$, then $S'$ can't be a point.

    \item Assume $X=\RR^n$, equipped with the Euclidean metric, and let $t$ be the longest lifespan of a finite bar in $PD(S)$. If $\e < t/2$, then $S'$ can't be star-shaped (and in particular, it can't be convex).
\end{enumerate}
\end{proposition}

\begin{proof}
    (1) This is a direct consequence of Theorem \ref{ThmStab}.

    (2) Assume $S'$ is the disjoint union of its subsets $A$ and $B$, with $\inf_{a\in A, b \in B}d(A,B)> 2t + 2\e$. Define $S_A=\{s\in S; \exists a \in A: d(s,a)\leq \e \}$ and $S_B=\{s\in S; \exists b \in B: d(s,b)\leq \e \}$. Clearly $S$ is the disjoint union of $S_A$ and $S_B$ and in fact, for any $s_A\in S_A$ and $s_B \in S_B$ we have $d(s_A,s_B)>2t$ by the assumption. This implies that $S_A^t$ (i.e., the closed $t$-neighborhood of $S_A$) and $S_B^t$ are disjoint components of $S^t$ and hence $\Cech(S,t)\simeq S^t$ is not connected. Thus the largest death time $T$ of $H_0$ is larger than $t$, as $\Cech(S,r)$ is connected for $r \geq T$.

    (3) Assume $S'$ is path connected. We will prove that if $d_H(S,S')<t$, then the largest finite death time of a 0-dim bar of $PD(S)$ is smaller than $t$. 
    
    We claim the following: Assuming $d_H(S,S')<t'<t$, for any pair of points $s,s'\in S$ there exists a sequence of points $s=s_0, s_1, \ldots, s_k=s'$ in $S$, such that $d(s_i, s_{i+1})<2t'$ (this implies that for each $i$, $(s_i,s_{i+1})$ is an edge in $\Cech(S,t')$). The claim implies that each pair of points of $S$ is connected by path in $S^{t'} \simeq \Cech(S,t')$, and thus the largest finite death time of a 0-dim bar of $PD(S)$ is at most $t'$ (hence is smaller than $t$). 

    The proof of the claim: Choose $x,x'\in S'$ such that $d(s,x)<t'$ and $d(s', x')<t'$. Choose a path $\gamma\colon [0,1]\to S'$ between $x=\gamma(0)$ and $x'=\gamma(1)$. As the open $t$-balls around the points of $S$ cover $S'$, we can find a sequence of points $s=s_0, s_1, \ldots, s_k=s'$ in $S$, and a sequence of parameters $0=v_0 < v_1< \ldots <v_k< v_{k+1=1}$ in $[0,1]$, such that for each $i$ the open $t'$-ball around $s_i$ contains $\gamma([v_i, v_{i+1}])$. The claim is proved by observing that $d(s_i, s_{i+1}) \leq d(s_i, \gamma(v_{i+1})) + d(\gamma(v_{i+1}), s_{i+1})< t' + t' = 2t'$.

    (5) Assume $S'=\{x\}$. As $\e< t/2$, $S$ lies within the open $t/2$-ball around $x$, hence the pairwise distances within $S$ are less than $t$, a contradiction to our observation. 

    (6) According to Lemma \ref{LemmaStar}, a star-shaped set has trivial persistence diagram, i.e., the only non-trivial bar is $[0,\infty)$ in dimension $0$, hence the Hausdorff distance from $S$ to such a set is at least $t/2$ by the Stability Theorem.
\end{proof}

A subset $A \subset \RR^n$ is \textbf{star-shaped} if there exists $a\in A$, such that $\forall x\in A$, the line segment from $x$ to $a$ lies in $A$. Each convex set is star-shaped. 


\begin{lemma}
\label{LemmaStar}
    Given a star-shaped metric space $A\subset \RR^n$, $\Cech(A,r)$ is contractible for each $r>0$, and thus $PD(A)$ contains exactly one bar: $[0,\infty)$ in dimension zero. 
\end{lemma}

\begin{proof}
    The proof is inspired by that of Hausmann and \cite{ZVFoot}. We will use the Whitehead theorem and prove that all the homotopy groups of $\Cech(A,r)$ are trivial. Without loss of generality we may assume $a=(0, \ldots, 0)=0^n\in \RR^n$, and we use $a$ as our basepoint.

    First, it is easy to see that $\Cech(A,r)$ is connected as $A$ is path connected. 

    In order to prove that $\pi_k(\Cech(A,r))$ is trivial, take a simplicial sphere $f\colon S^k \to \Cech(A,r)$ based at $0^n \in \RR^n$. Choose $\e>0$ so that for each simplex $\sigma$ of the chosen triangulation of $S^n$, the intersection $\cap_{v\in \sigma}B_{\RR^n}(v,r-\e)\neq \emptyset$, i.e., so that the image of $f$ lies in $Cech(A, r-\e)$. Choose $\mu<1$ so that for each vertex $v$ of $S^n$, $d(f(v), \mu \cdot f(v))<\e$ (here and in similr expressions involving the multiplication with $\mu$, we treat $f(v)$ as a point in $\RR^n$.) [we could express $\mu$ explicitly using the diameter of the image of $f$, but we don't need an explicit expression]. The star-shaped assumption implies that for each vertex $v$ of $S^n$, $\mu \cdot f(v)\in A$.  We next prove that maps $f$ and $\mu \cdot f$ are homotopic in $\pi_k(\Cech(A,r))$.
    
    For a  simplex $\sigma \in S^n$, observe that $\cap_{v\in \sigma} B(f(v), r)$ contains a point $c_\sigma$. Then $\mu \cdot f(\sigma)$ is a simplex in $\Cech(A,r)$ as $\cap_{v\in \sigma} B(\mu \cdot f(v), r)$ contains a point $\mu \cdot c_\sigma$, thus $\mu \cdot f$ is a well defined simplicial map into $\Cech(A,r)$. Furthermore, $d(c_\sigma,\mu \cdot c_\sigma )<\e$. Hence 
    \[
    \cap_{v\in \sigma} B(f(v), r) \bigcap \cap_{v\in \sigma} B(\mu \cdot f(v), r)
    \]
    contains $c_\sigma$, meaning that $f(\sigma) \cup \mu\cdot f(\sigma)$ forms a simplex in $\Cech(A,r)$. This means that $f$ and $\mu \cdot f$ are contiguous and have the same basepoint $0^n = \mu \cdot 0^n$, thus they are based homotopic in $\Cech(A,r)$. 

    We proceed inductively with the same constant $\mu$ to deduce that all $\mu^j \cdot f$ are homotopic. For large $j$, the vertices of the image of $\mu^j \cdot f$ are contained $B(0^n, r)$, which spans a full simplex in $\Cech(A,r)$, hence $\mu^j \cdot f$, and by extension $f$, are null-homotopic.
\end{proof}

Given $t>0$ and $A\subset \RR^n$, define $A^t=\cup_{a \in A} \overline{B(a,x)}$ to be the $t$-thickening of $A$, with $\overline{B(a,x)}$ being the closed $r$ ball around $x$. The following statement is a \textit{reconstruction result}, stating that if $X$ is tame enough and $S$ is close enough to $X$, then the Cech complexes of $S$ at certain scales reveal the homotopy type of $X$.

\begin{theorem} \label{Recon1}
    [Proposition 5 of \cite{Wint}, Reconstruction for sets of positive reach]
    Let $X \subset \RR^n$ be a subset of positive reach and choose $\R \leq  \rch(X)$. Assume that $S\subset \RR^n$ is a sample approximating $X$, so that for parameters $\e, \delta < \R$:
    \begin{enumerate}
        \item $X \subseteq S^\e$, i.e.,  for each point of $X$ there is at least one point of the sample at distance at most $\e$ from it. 
        \item $S \subseteq X^\delta$, i.e., each point of the sample is at most $\delta$ far from $X$. In particular, $d_H(X,S) \leq \max\{\e, \delta\}$.
    \end{enumerate}
    If $\e$ and $\delta$ satisfy
    \[
    \e + \delta \sqrt{2} \leq (\sqrt{2}-1)\R,
    \]
    then $\Cech(S,r)\simeq X$ for each 
    \[
    r \in \left[ 
    \frac{1}{2}\left(\R + \e - \sqrt{\Delta}\right), 
    \frac{1}{2}\left(\R + \e + \sqrt{\Delta}\right)
    \right],
    \]
    where $\Delta = 2(\R -\delta)^2 - (\R + \e)^2$.
\end{theorem}

More suitable for our purposes is the following:

\begin{corollary}
\label{CorRecon1}
    Let $X \subset \RR^n$ be a subset of positive reach and let $S\subset \RR^n$ be a sample of $X$. Assume the Cech persistent homology of $S$ contains no bars being born or dying on the interval $[a,b]$. If:
    \begin{enumerate}
        \item $\rch(X) \geq a/(2-\sqrt{2})$, and
        \item $d_H(S,X)\leq  (3-2\sqrt{2})\rho$, where $\rho:= \min \left\{ 
        \rch(X),b/(2-\sqrt{2})
        \right\},$
    \end{enumerate}
    then $\rho(2-\sqrt{2}) \in [a,b]$ and $\Cech(S,\rho(2-\sqrt{2}))\simeq X$. In particular, the PH of the sample set $S$ on $[a,b]$ reveals the homology of $X$. 

    Additionally, if Cech-PH contains a birth or death on each of the $[a',a], [b, b']$ for some $[a',b'] \supset [a,b]$, then
    \[
    \rch(X)+d_H(S,X) - \sqrt{2(\rch(X) -d_H(S,X))^2 - (\rch(X) + d_H(S,X))^2} > 2a',
    \]
    \[
    \rch(X)+d_H(S,X) + \sqrt{2(\rch(X) -d_H(S,X))^2 + (\rch(X) + d_H(S,X))^2} < 2b'.
    \]
\end{corollary}

\begin{proof}
    A verification of Theorem \ref{Recon1}. If we choose $\e=\delta=(3-2\sqrt{2})\R$, it is easy to verify that in the notation of Theorem \ref{Recon1}, (1) holds with equality and $\Delta=0$. The Theorem thus implies that $\Cech(S,(\R+\e)/2) \simeq X.$

    Now set $\R=\rho$ and $\e=(3-2\sqrt{2}) \rho$ and observe:
    \begin{itemize}
        \item $\R \leq \rch(X)$ as required by the Theorem.
        \item $d_H(S,X)\leq \e$ as required by the Theorem.
    \end{itemize}

    Thus the reconstruction scale $(\R+\e)/2$ for which the Theorem implies $\Cech(S,(\R+\e)/2) \simeq X$ satisfies:
    \[
    (\R+\e)/2 =(\R + (3-2\sqrt{2})\R)/2=\R(2-\sqrt{2})=\rho(2-\sqrt{2}), 
    \]
    which is at least $a$ by (1) and (2), and at most $b$ by (2).

    The additional conclusion follows from the fact that the reconstruction interval of the Theorem has to have the endpoints contained in the interior of $[a', b']$.
\end{proof}

If we want to extract the homology of a Euclidean subset $X$ from a sample $S$, as the collection of long bars of $PD(S)$, we may use Theorem 3.5 of \cite{CO}. That is, by the stability result, if all the bars of PD of $X$ are of length more than $4\e$, and the sample $S$ satisfies $d_H(X,S)<\e$, then the bars of $PD(X)$ are exactly the bars of $PD(S)$ that are of length more than $2\e - d_H(X,S)$, with each of the bar endpoints being changed by at most $\e$.

However, in general it is hard to put a reasonable assumption on the length of the bars in a PD. The statement in \cite{CO} is phrased in terms of the so-called weak feature size ($wfs$). Using the more restrictive, but geometrically simpler notion of reach introduced above, this statement reads as follows. We note that $\rch \leq wfs$.

\begin{theorem}\label{Recon3}
[Theorem 3.5 of \cite{CO}]
     Let $X \subset \RR^n$ be a compact subset  of positive reach. Assume that $S\subset \RR^n$ is a sample approximating $X$, so that $d_H(X,S) < \e < \frac{\rch(X)}{4}$. Then for each $\alpha, \alpha' \in [\e, \rch(X)-\e]$ with $\alpha' - \alpha > 2\e$, we have 
     \[
        H_k(X) \cong \im \Big(H_K(\Cech(S,\alpha))\to H_K(\Cech(S,\alpha'))\Big).
     \]
     In particular, $H_k(X)$ is represented by the bars of $PD(S)$ in dimension $k$ which contain the interval $[\alpha, \alpha']$. 
\end{theorem}

Theorem \ref{Recon3} demonstrates how to extract the homology of $X$ from the homology of $PH(S)$ persisting through the interval $(\alpha, \alpha')$, i.e.,  by ignoring the bars that do not contain the interval $(\alpha, \alpha')$.
The following is a rephrasing suited to our setup.

\begin{corollary}
\label{CorRecon3}
Let $X \subset \RR^n$ be a compact subset  of positive reach. Assume that $S\subset \RR^n$ is a sample approximating $X$ and that for some positive $\alpha < \alpha'$ we have 
\[
       \widetilde H = \im \Big(H_K(\Cech(S,\alpha))\to H_K(\Cech(S,\alpha'))\Big).
     \]
     If 
     \[
    d_H(X,S)<\min \{\alpha, (\alpha' - \alpha)/2\}, \quad \textrm{ and } \quad \rch(X)> \max \{ 4 d_H(X,S), \alpha' + d_H(X,S)\},
     \]
     then $H_*(X)\cong \widetilde H$.
\end{corollary}

\begin{proof}
    We verify that the assumptions of Theorem \ref{Recon3} indeed hold. The assumptions on the parameters of the theorem are:
    \begin{enumerate}
        \item $d_H(X,S) < \e < \frac{\rch(X)}{4}$
        \item $\alpha, \alpha' \in [\e, \rch(X)-\e]$
        \item $\alpha' - \alpha > 2\e$.
    \end{enumerate}
    We have to show that under the assumptions of the Corollary, a suitable choice of $\e$ exists. 

    Any $\e \in (d_H(X,S), \min \{\alpha, (\alpha' - \alpha)/2\})$ satisfies (3) and the first parts of (1) and (2). We can decrease $\e$ towards $d_H(X,S)$ until $\rch(X)> 4\e$, as $\rch(X)> 4d_H(X,S)$, thus confirming (1). We can further decrease $\e$ towards $d_H(X,S)$ until $\rch(X)> \alpha' +\e$, as $\rch(X)> \alpha' +d_H(X,S)$, thus confirming (2).
\end{proof}

Observe that the results above ensure that one can read the homology of a set $X$ by studying the PD of a sample $S$ of it, as long as they are close enough. They depend on $\rch{X}$ and thus it is useful to have a lower bound for this constant. It turns out that such a bound can be found, at least in the case where $X$ is an invariant manifold of a vector field.

\begin{proposition} \label{reach proposition}
Let $X \subset \mathbb{R}^n$ be a compact, smooth submanifold of dimension $m$, which is invariant under the flow of the $C^1$ vector field $F: \mathbb{R}^n \to \mathbb{R}^n$ (i.e. $\forall p\in X,\ F(p)\in T_pX$). Let us suppose that, $\forall p\in X$, there exist $b,L > 0$ such that $\|F(p)\| \ge b$ and $\|DF(p)\|\le L$. Then, the inequality $\text{rch}(X) \ge \frac{b}{2L}$ holds.
\end{proposition}

\begin{proof}
Let $\rho_X$ be the maximal radius of curvature of $X$ and 
$$\eta_X=\frac{1}{2}\inf \{\|x-y\|,\ x,y\in X, x\neq y, y\in N_xX, x\in N_yX\},$$
where we denote by $NX$ the normal bundle of $X$. We then have $\text{rch}(X)=\min \{\rho_X,\eta_X\}$ (see \cite{Aamari}). We wish to estimate $\rho_M$. 

Let $p \in X$, $\xi \in N_pX$ with $\|\xi\|=1$ and consider the shape operator $A_{\xi}: T_pX \to T_pX$. Recall that the eigenvalues of $A_{\xi}$ are the principal curvatures of $X$, with respect to the $\xi$ direction. Furthermore, if $w \in T_pX$, it is:
$$ A_{\xi}(w) = -(d\xi \cdot w)^{\top}, $$
where $(\cdot)^{\top}$ stands for the orthogonal projection to $T_pX$.

The vector field $F$ is orthogonal to $\xi$. If $\gamma:I\subset \mathbb{R}\to X$ is a smooth curve, with $\gamma(0)=p$ and $\dot{\gamma}(0)=w$, $(\|w\|=1)$, we differentiate the relation 
$$ \langle F(\gamma(t)), \xi (\gamma(t)) \rangle = 0, $$
with respect to $t$ and evaluate at $t=0$. We get:
\begin{equation*}
\begin{split}
\langle d_pF \cdot w, \xi \rangle + \langle F(p), d_p\xi \cdot w \rangle &= 0\Rightarrow \langle d_pF \cdot w, \xi \rangle = - \langle F(p), (d_p\xi \cdot w)^{\top} \rangle \Rightarrow\\
\Rightarrow \langle d_pF \cdot w, \xi \rangle &=\langle F(p),A_{\xi}(w)\rangle.
\end{split}
\end{equation*}
We have used the equality $\langle F(p), d_p\xi \cdot w \rangle= \langle F(p), (d_p\xi \cdot w)^{\top} \rangle$, which holds because $F$ is tangent to $X$.

Now, if $\kappa$ is an eigenvalue of the shape operator, with corresponding eigenvector  $w \in T_pX$, $\|w\|=1$, we get the equality:
$$ \langle d_pF \cdot w, \xi \rangle = \langle F(p), \kappa w \rangle = \kappa \langle F(p), w \rangle. $$

We therefore have:
$$ |\kappa| \cdot |\langle F(p), w \rangle| = |\langle d_pF \cdot w, \xi \rangle| \le \|d_pF \cdot w\| \cdot \|\xi\| \le L. $$

The curvature $\kappa$ is thus bounded above by $\frac{L}{|\langle F(p), w \rangle|}$. The denominator achieves its maximal value when $w=\pm\frac{F(p)}{\|F(p)\|}$ (recall that $\|w\|=1$) and because of this,
$$|\kappa|\le \frac{L}{\|F(p)\|},$$
which gives us the desired inequallity:
$$ \rho_M = \frac{1}{\kappa_{\max}} \ge \frac{b}{L}. $$

On the other hand, according to \cite{Aamari}, it is $\eta _X\geq \frac{\rho_X}{2}$ and we get the conclusion. 
\end{proof}

We shall use these results, in what follows, to study the limit sets of specific dynamical systems. 
\section{Applications}
In this section we apply the machinery developed above to study limit sets of specific dynamical systems. Under mild assumptions on the reach of the limit set, we prove that these sets have specific geometric properties. In this way, we can prove the existence of limit cycles and tori, in the corresponding phase spaces.
\subsection{Limit cycle in the van der Pol oscillator}
Consider the following system:
\begin{eqnarray}
\begin{cases}
\label{van der pol system}
\dot{x}=y \\
\dot{y}=\mu (1-x^2)y-x
\end{cases},
\end{eqnarray}
which is the well known van der Pol system. For $\mu=0.1$, this system possesses a stable limit cycle that surrounds the origin, which is an unstable fixed point. In figure \ref{fig1}, we present the phase portrait of system (\ref{van der pol system}). 
\begin{figure}[h]
\begin{center}
\includegraphics[scale=0.4]{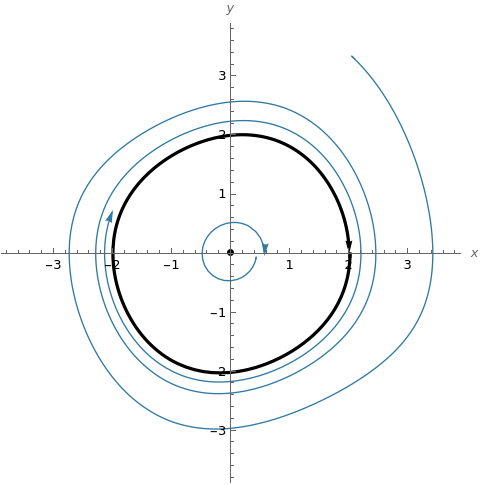}
\caption{Numerically obtained solutions for the van der Pol system: the unstable fixed point at the origin, the sought after periodic orbit (thick), an orbit which spirals away from the fixed point and a third orbit, which comes from infinity. The last two orbits tend to the periodic orbit.}
\label{fig1}
\end{center}
\end{figure}
A proof for the existence of the stable periodic orbit, depicted in figure \ref{fig1}, can be found in \cite{Hirsch-Smale-Devaney}.

We apply to system (\ref{van der pol system}) the methodology described above as follows.
\begin{enumerate}
\item {We numerically integrate the system (\ref{van der pol system}), with initial conditions $(x_0,y_0)=(0,10)$, for $t\in [0,1000]$. We depict this solution, for $t\in [900,1000]$, in figure \ref{fig2}. This is an orbit segment which, as described in section 2, we assume to be close to the $\omega$-limit of the orbit, in the sense of Hausdorff distance.
\begin{figure}[h]
\begin{center}
\includegraphics[scale=0.3]{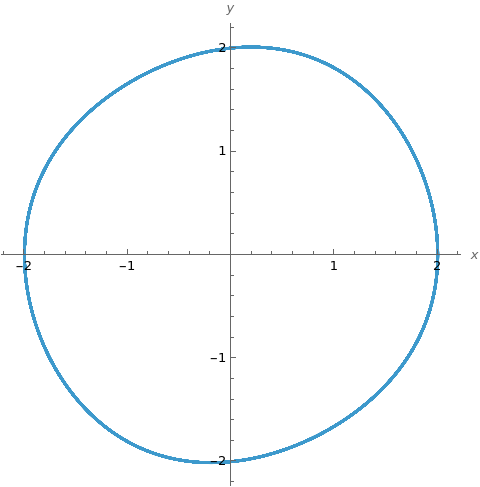}
\hspace*{2.cm}
\includegraphics[scale=0.3]{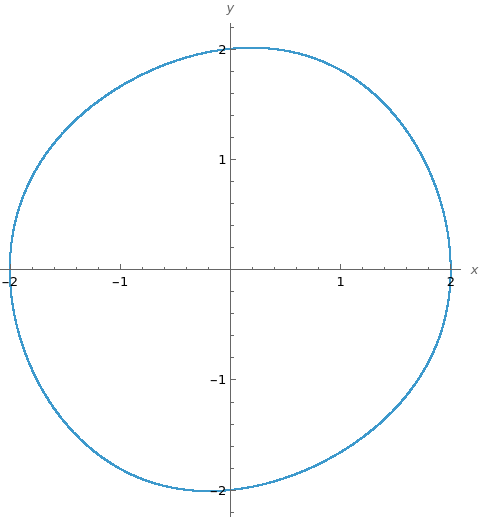}
\caption{Left: Numerically obtained solution for system (\ref{van der pol system}), with initial conditions $(0,10)$. Here $t\in [900,1000]$. Right: The sample set of the previous solution. For details, see the text.}
\label{fig2}
\end{center}
\end{figure}
}
\item {We sample points from this numerical solution: we recorded every point of it, from $t=100$ to $t=1000$, with a time-step of $0.1$. This set, S, of points is also depicted in figure \ref{fig2}.
}
\item {The minimum distance between these points of S was found to be $0.0017...$, while the maximum was $4.13...$. Thus we computed the persistent diagram for this set, using the Cech complex, with $\epsilon\in [0,4]$. The result is shown in figure \ref{fig4}.
\begin{figure}[h]
\begin{center}
\includegraphics[scale=0.5]{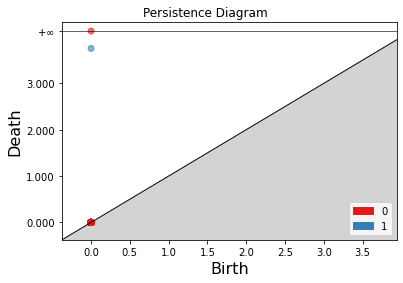}
\caption{Persistence diagram, for the sample set S depicted in figure \ref{fig2}, obtained via the Cech complex.}
\label{fig4}
\end{center}
\end{figure}
}
\end{enumerate}
In this persistence diagram, we observe the following.
\begin{enumerate}
    \item The interval of persistence of the $0$-dim bar is $(0,\infty)$.
    \item The interval for the $1$-dim bar is $(a,b)$, where $a\leq 0.1, b \geq 3.5$.
    \item All the other bars die before $0.1$.
\end{enumerate}
Given that the involved Cech complexes are homotopy equivalent to finite planar simplicial complexes, we conclude that the inclusions of the Cech complexes are homotopy equivalences of $\mathbb{S}^1$ on the interval $[0.1,3.5]$.

\begin{proposition}
\label{Prop2Da}
Let $X$ be the $\omega$-limit set of the orbit of system (\ref{van der pol system}) having as initial condition the point $(0,10)$ and S the sample point set described above. Assume that $\rch(X)>0$ and $d_H(X,S)\leq \e$.  
\begin{enumerate}
    \item The bottleneck distance between the geometric descriptors $PD(X)$ and $PD(S)$ is at most $\e$.
    
    \item Let $t<\infty$ be the largest finite death time of a 0-dim bar of $PD(S)$. Then $X$ cannot be represented as the disjoint union of two subspaces which are more than $2t + 2\e$ apart. 
    
    \item Let $t<\infty$ be the largest finite death time of a 0-dim bar of $PD(S)$. If $X$ is path-connected, then $d_H(S,X)\geq t$.
    
    \item If we further assume that $\rch(X)\geq \frac{0.1}{2-\sqrt{2}}\approx 0.17071$ and             \[
            \e \leq \min\left\{(3-2\sqrt{2})\rch(X),(3-2\sqrt{2}) \frac{3.5}{2-\sqrt{2}} \right\}\approx \min\left\{0.17157\rch(X),1.025126 \right\},
            \]
then $X\simeq \mathbb{S}^1$, i.e., the limit set is a circle up to homotopy.

    \item Let $t=4.13$ be the largest distance between the points of $S$. If $\e< 4.12/2$, then the limit set can't be a point.

    \item If $\e < \frac{3.5-0.1}{2}$, then $X$ can't be star-like (and in particular, it can't be convex).
    
\end{enumerate}
\end{proposition}

\begin{proof}
    (4) By Corollary \ref{CorRecon1}, $X$ is homotopy equivalent to the Cech complex of $S$ at scale from $[0.1, 3.5]$, which, by the nerve theorem, is homotopy equivalent to a connected planar simplicial complex with the first Betti number being $1$, hence it is $\mathbb{S}^1$ up to homotopy.

    All the other items are direct applications of Corollary \ref{CorStab}.
\end{proof}

Let us focus our attention to point (4) of proposition above. By using proposition \ref{reach proposition}, one can show that every invariant set of the system (\ref{van der pol system}) contained in the annulus $\{(x,y)\in \mathbb{R}^2, 1\leq x^2+y^2\leq 9\}$ must have reach greater than $0.22$, thus the first assumption holds. The second assumption ensures that $S$ is a ``good enough" sample, in the sense that the orbit segment we sampled indeed has approached its limit set. Since system (\ref{van der pol system}) has a unique fixed point, located at the origin, while the $\omega$-limit set of any orbit of any (smooth) planar system cannot be an annulus, we can state the following:
\begin{corollary}
    Let $X$ be the $\omega$-limit set of the orbit of system (\ref{van der pol system}) having as initial condition the point $(0,10)$ and S the sample point set described above. Assume that $d_H(X,S)\leq 0.037$. Then $X$ is a closed orbit. In particular, the system (\ref{van der pol system}) possesses a periodic orbit, surrounding the origin.
\end{corollary}

Of course, the existence of this limit cycle was known. We turn now our attention to another planar system.
\subsection{Another stable cycle on the plane}

Consider the system:
\begin{eqnarray}
\begin{cases}
\label{2nd plane system}
\dot{x}=1-x^2-y^2-2y\\
\dot{y}=1-x^2-y^2+2x
\end{cases},
\end{eqnarray}
which possesses:
\begin{itemize}
\item {two fixed points: one is located at $(\frac{1-\sqrt{3}}{2},\frac{-1+\sqrt{3}}{2})$ and it is an unstable spiral, while the other one is a saddle point, located at $(\frac{1+\sqrt{3}}{2},\frac{-1-\sqrt{3}}{2})$.}
\item {a stable limit cycle, surrounding the fixed point in the second quadrant of the plane, as numerical simulations imply. To the best of our knowledge, a proof of the existence of this cycle does not exist.}
\end{itemize}
We present, in figure (\ref{fig5}), the phase portrait of this system. 
\begin{figure}[h]
\begin{center}
\includegraphics[scale=0.5]{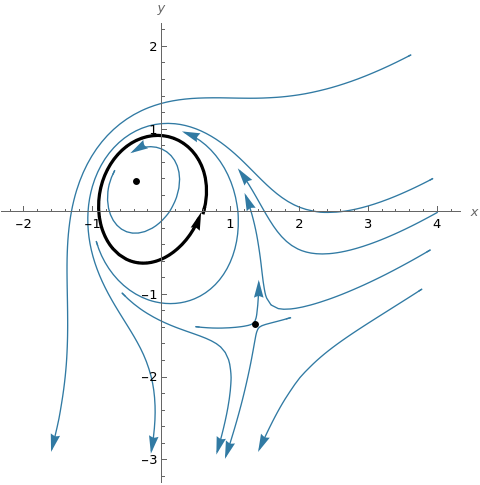}
\caption{Phase portrait for system (\ref{2nd plane system}): the two fixed points and a numerically obtained limit cycle are crearly seen.}
\label{fig5}
\end{center}
\end{figure}

To study this system, we proceed as follows.
\begin{enumerate}
\item {We numerically integrate system (\ref{2nd plane system}), with initial conditions $(x_0,y_0)=(0,0.2)$, while $t\in [0,1000]$. We depict this solution, for $t\in [800,1000]$, in figure (\ref{fig6}).
\begin{figure}[h]
\begin{center}
\includegraphics[scale=0.3]{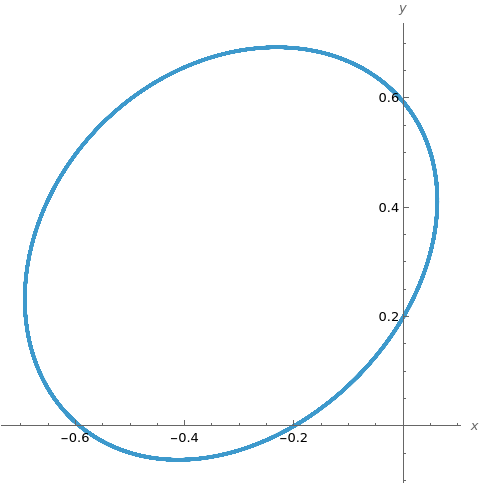}
\hspace*{2.cm}
\includegraphics[scale=0.3]{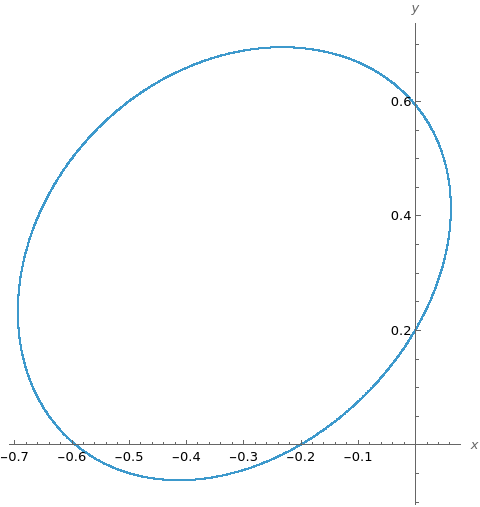}
\caption{Left panel: numerically obtained solution of system (\ref{2nd plane system}), with initial conditions $(0,0.2)$Here, $t\in [800,1000]$. Right panel: Sample points from this numerically obtained orbit. Here $t\in[100,1000]$ with a time-step of $0.1$.}
\label{fig6}
\end{center}
\end{figure}
}
\item {We sample points from this numerical solution: we record every point of it, from $t=100$ to $t=1000$, with a time-step of $0.1$. This set $S$ is also depicted in figure (\ref{fig6}).
}
\item {The minimum distance between the points of $S$ was found to be $0.0066..$, while the maximum distance was $0.841...$. Thus we compute the persistent diagram for this set, with $\epsilon\in [0,1]$. The result is shown in figure (\ref{fig8}).
\begin{figure}[h]
\begin{center}
\includegraphics[scale=0.5]{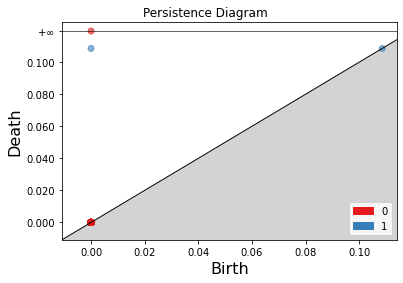}
\caption{The persistence diagram for the sample points of the previous figure.}
\label{fig8}
\end{center}
\end{figure}
}
\end{enumerate}

In this persistence diagram, we observe the following.
\begin{enumerate}
    \item The interval of persistence of the $0$-dim bar is $(0,\infty)$.
    \item The interval for the $1$-dim bar is $(a,b)$, where $a\leq 0.004, b \geq 0.105$.
    \item All the other bars die before $0.004$, except for the $1$-dim bar being born after $0.105$.
\end{enumerate}
Given that the involved Cech complexes are homotopy equivalent to finite planar simplicial complexes, we conclude that the inclusions of the Cech complexes are homotopy equivalences of $\mathbb{S}^1$ on the interval $[0.004,0.105]$.

\begin{proposition}
\label{Prop2Db}
Let $X$ be the $\omega$-limit set of the orbit of system (\ref{2nd plane system}) with initial condition $(0,0.2)$ and S the sample point set described above. Assume that $\rch(X)>0$ and $d_H(X,S)\leq \e$.  
\begin{enumerate}
    \item The bottleneck distance between the geometric descriptors $PD(X)$ and $PD(S)$ is at most $\e$.
    
    \item Let $t<\infty$ be the largest finite death time of a 0-dim bar of $PD(S)$. Then $X$ cannot be represented as the disjoint union of two subspaces which are more than $2t + 2\e$ apart. 
    
    \item Let $t<\infty$ be the largest finite death time of a 0-dim bar of $PD(S)$. If $X$ is path-connected, then $d_H(S,X)\geq t$.
    
    \item If we further assume that $\rch(X)\geq \frac{0.004}{2-\sqrt{2}}\approx 0.00682843$ and             \[
            \e \leq \min\left\{(3-2\sqrt{2})\rch(X),(3-2\sqrt{2}) \frac{0.105}{2-\sqrt{2}} \right\}\approx \min\left\{0.17157\rch(X),0.0307538 \right\},
            \]
then $X\simeq \mathbb{S}^1$, i.e., the limit set is a circle up to homotopy.

    \item Let $t=0.841$ be the largest distance between the points of $S$. If $\e< 0.840/2$, then the limit set can't be a point.

    \item If $\e < \frac{0.105-.004}{2}=0.0505$, then $X$ can't be star-like (and in particular, it can't be convex).
    
\end{enumerate}
\end{proposition}

The proof is analogous to that of Proposition \ref{Prop2Da}.

Once again, we focus our attention to the fourth point of the above proposition. The set $S$ is contained in the region
\[
\left\{(x,y)\in \mathbb{R}^2, \ 0.16<\left(x-\frac{1-\sqrt{3}}{2}\right)^2+\left(y-\frac{-1+\sqrt{3}}{2}\right)^2<1.4
\right\},
\]
and one can show, using proposition \ref{reach proposition}, that every invariant set of system (\ref{2nd plane system}) contained in this region has reach bounded from below from the number $0.05$, thus the first assumption holds. Also, the limit set cannot be an annulus. We therefore have the following:
\begin{corollary}
Let $X$ be the $\omega$-limit set of the orbit of system (\ref{2nd plane system}) with initial condition $(0,0.2)$ and S the sample point set described above. Assume that $d_H(X,S)\leq 0.009$. Then $X$ is hommeomorphic to a cycle. In particulat, system (\ref{2nd plane system}) possesses a periodic orbit.
\end{corollary}

We now turn our attention to limit sets more complicated than periodic orbits.
\subsection{An invariant torus in $\RR^3$}
In \cite{Jafari-Sprott}, the following system was presented:
\begin{eqnarray}
\label{1st torus}
\begin{cases}
\dot{x}= y\\
\dot{y}=  -x - yz\\
\dot{z}= y^2 - a + bz
\end{cases},
\end{eqnarray}
where $a=7$ and $b=0.55$. This system has a number of interesting properties, among them being the apparent existence of an attracting torus, which we depict in figure \ref{fig9}.
\begin{figure}[h]
\begin{center}
\includegraphics[width=10cm,height=6cm]{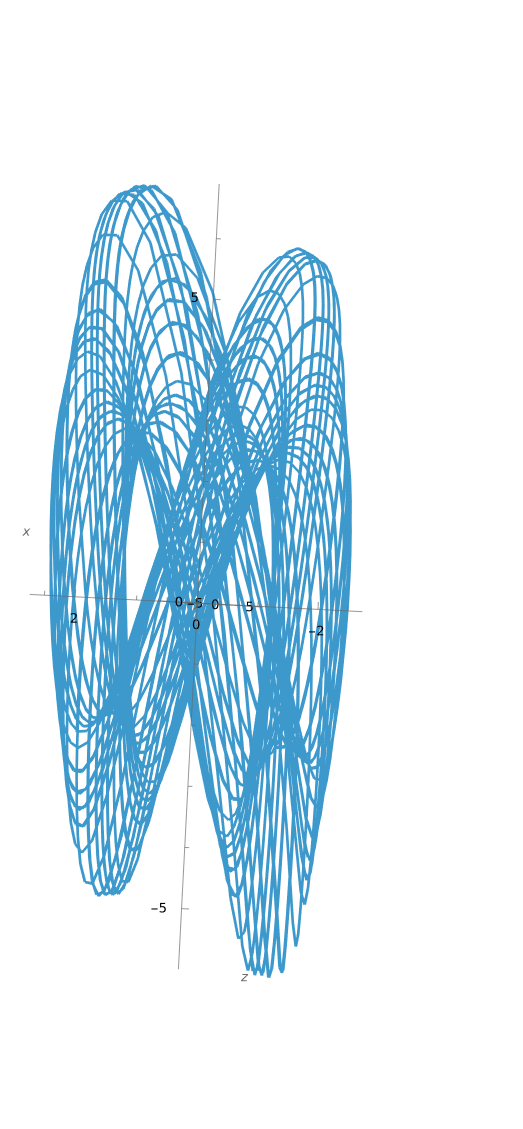}
\caption{The attracting torus of system (\ref{1st torus}). Here, the initial condition is $(0.5,0.5,0)$ and the system was numerically integrated for $t\in [0,10000]$. We present only the $t\in [9800,10000]$ part of the obtained orbit. The coordinate axes are arranged so as the shape of the torus is clearly seen.}
\label{fig9}
\end{center}
\end{figure}

\begin{figure}[h]
\begin{center}
\includegraphics[width=8cm,height=5cm]{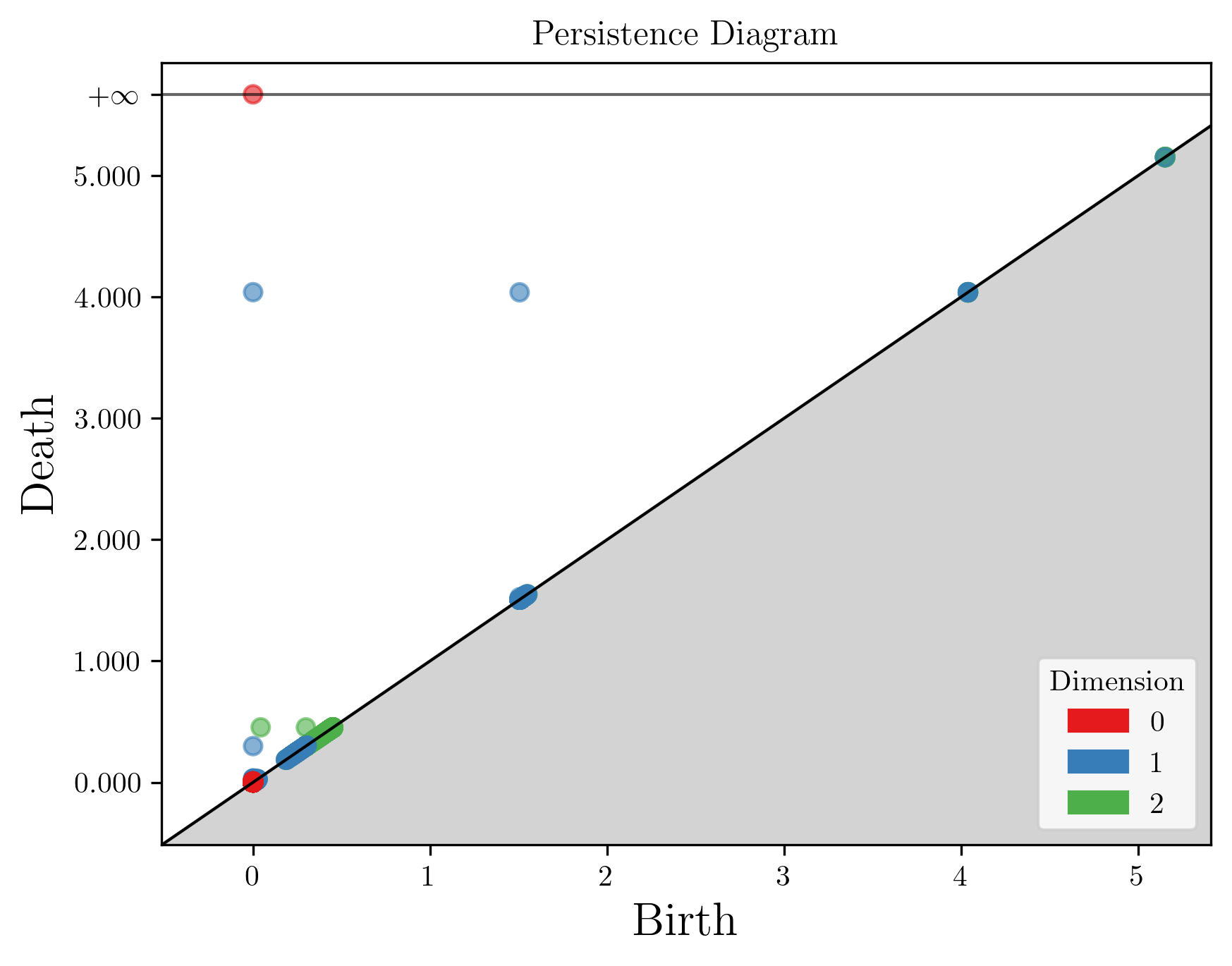}
\caption{Persistence diagram for the sample point set.}
\label{fig11}
\end{center}
\end{figure}

This torus is the $\omega$-limit set of various orbits of system (\ref{1st torus}). To produce a point-set which will densely cover this torus, we pick three initial conditions, namely $(0.5,0,0.5),\ (0.5,0,0),\ (1,6,0)$ and record their orbits in the time interval $t\in [5000,10000]$, with a time-step of $0.1$. As before, we name $S$ this sample set, while $X$ stands for the actual $\omega$-limit set of these orbits, which we assume that has positive reach, while $d_H(X,S)<\e$. The persistence diagram of $S$ is shown in figure \ref{fig11}.

We observe that, for $r\in [a,b]$, where $a=0.043295$ (=the birth of the first 2-dim class, as the last death of the first cluster along the diagonal is smaller) and $b=0.185908$ (=the first birth on the second cluster on the diagonal) the inclusions of $\Cech(S,r)$ induce the homology isomorphism, with the homology being that of a torus. 

\begin{proposition}
\label{Prop3Da}
Let us suppose that the orbits of system (\ref{1st torus}) with initial conditions $(0.5,0,0.5),\ (0.5,0,0),\ (1,6,0)$ have a common $\omega$-limit set $X$ and S the sample point set described above. Assume that $\rch(X)>0$ and $d_H(X,S)\leq \e$.  
\begin{enumerate}
    \item The bottleneck distance between the geometric descriptors $PD(X)$ and $PD(S)$ is at most $\e$.
    
    \item Let $t<\infty$ be the largest finite death time of a 0-dim bar of $PD(S)$. Then $X$ cannot be represented as the disjoint union of two subspaces which are more than $2t + 2\e$ apart. 
    
    \item Let $t<\infty$ be the largest finite death time of a 0-dim bar of $PD(S)$. If $X$ is path-connected, then $d_H(S,X)\geq t$.

    \item Let $t=14.9$ be the largest distance between the points of $S$. If $\e< 14.9/2$, then the limit set can't be a point.

    \item If $\e < \frac{4.0402920-0.0005142}{2}=2.01989$, then $X$ can't be star-like (and in particular, it can't be convex).
\end{enumerate}
\end{proposition}

\begin{proof}
    The statements are a direct consequence of Corollary \ref{CorStab}, with (5) using the fact that 
    \[(0.0005142, 4.0402920)\] 
    is the homology bar of the longest lifespan.
\end{proof}

\begin{remark}
    We can draw further geometric observations based on the fact that Cech complexes $\Cech(S,r)$ is homotopy equivalent to the $r$-neighborhood $S^r$ by the nerve theorem, and that $X^{r-\e}\subseteq S^r \subseteq X^{r+\e}$. $X^{0.0005142 + \e}$ contains a loop which only gets filled once we thicken $X$ by more than $4.0402920-\e$. At around $X^{1.5062854+ \e}$ another loop emerges which gets filled once we thicken $X$ by more than $4.0391847-\e$. These two loops getting filled at the same time might be the result of the reversing symmetry $(x,y,z)\mapsto (-x,-y,z)$ of the system.
\end{remark}

Let us now study the topological type of $X$.

\begin{proposition}
\label{Prop3Db}
Let us suppose that the orbits of system (\ref{1st torus}) with initial conditions $(0.5,0,0.5),\ (0.5,0,0),\ (1,6,0)$ have a common $\omega$-limit set $X$ and S the sample point set described above. Assume that $\rch(X)>0$ and $d_H(X,S)\leq \e$ satisfy the following:
\begin{itemize}
        \item Geometry assumption: assume $\rch(X)\geq \frac{0.043295}{2-\sqrt{2}}\approx 0.0739092$. 
        \item Sampling density assumption: Assume 
            \[
            \e \leq \min\left\{(3-2\sqrt{2})\rch(X),(3-2\sqrt{2}) \frac{0.185908}{2-\sqrt{2}} \right\}\approx \min\left\{0.17157\rch(X),0.0544512 \right\}.
            \]
    \end{itemize}
Then we have $X\simeq \Cech\Big(S,\min\left\{ \rch(X)0.585786, 0.185908) \right\}\Big)$, and thus $X$ has the homology groups of a torus.

Furthermore,  \[
    \rch(X)+d_H(S,X) - \sqrt{2(\rch(X) -d_H(S,X))^2 - (\rch(X) + d_H(S,X))^2} > 2\cdot 0.043294,
    \]
    \[
    \rch(X)+d_H(S,X) + \sqrt{2(\rch(X) -d_H(S,X))^2 + (\rch(X) + d_H(S,X))^2} < 2 \cdot 0.185909.
    \]
\end{proposition}

\begin{proof}
    The result is a direct application of Corollary \ref{CorRecon1}. For the first part (homotopy equivalence) we use the fact that on the scale interval $[a,b]=[0.043295,0.185908]$ the persistent homology of $S$ is constantly that of a torus. 

    For the second conclusion we use the fact that the mentioned homology of the torus on $[a,b]$ changes below $a=0.043295$ and above $b=0.185908$. So let's set $a'=0.043294$ and above $b'=0.185909$. Then the last conclusion of Corollary \ref{CorRecon1} implies that the pair $\rch(X),d_H(X,S))$ lies in the stated region, drawn on the right side of Figure \ref{ReachVSdH3D}.
\end{proof}

\begin{figure}[h]
\begin{center}
\includegraphics[scale=0.3]{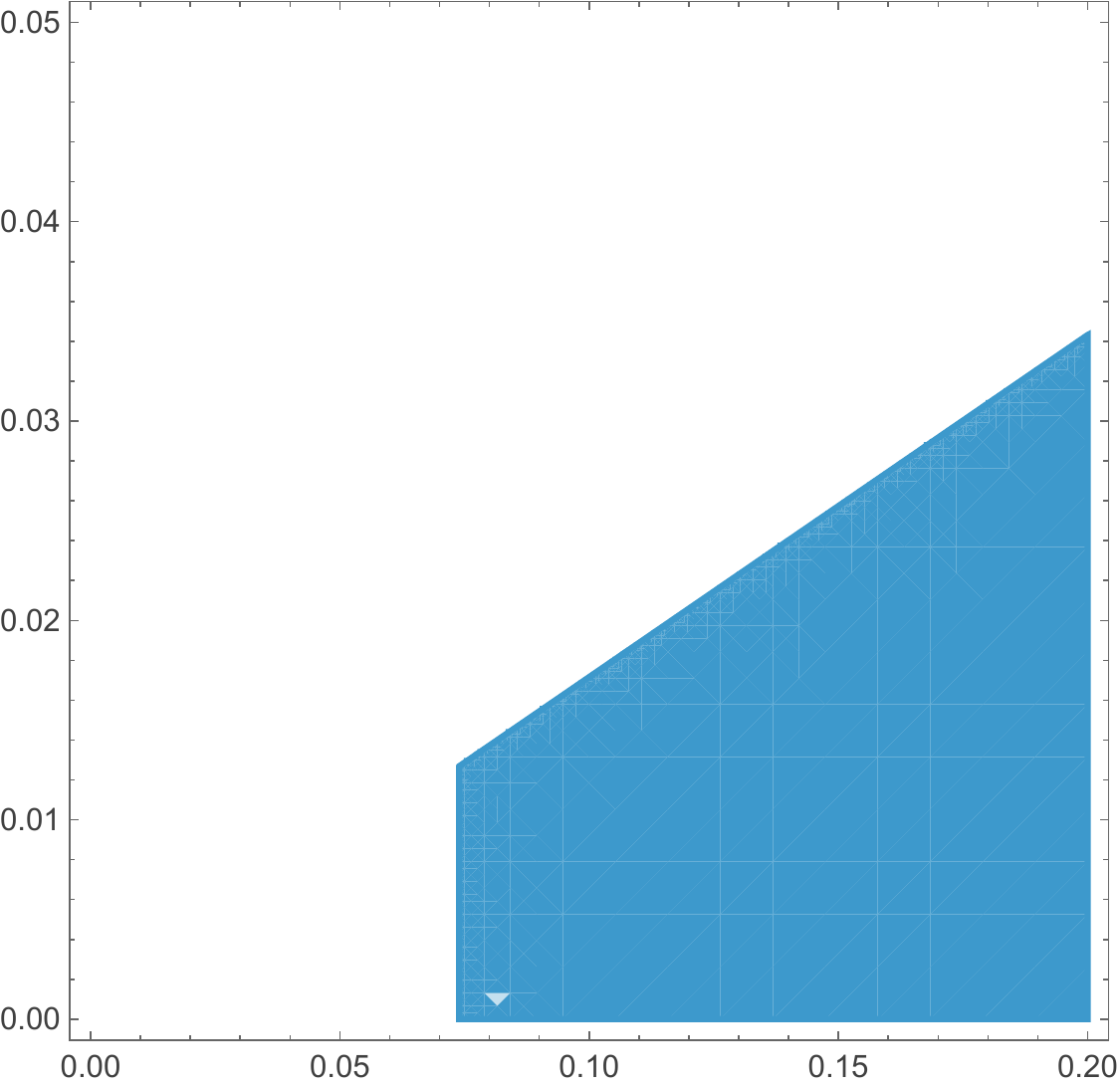}
\includegraphics[scale=0.3]{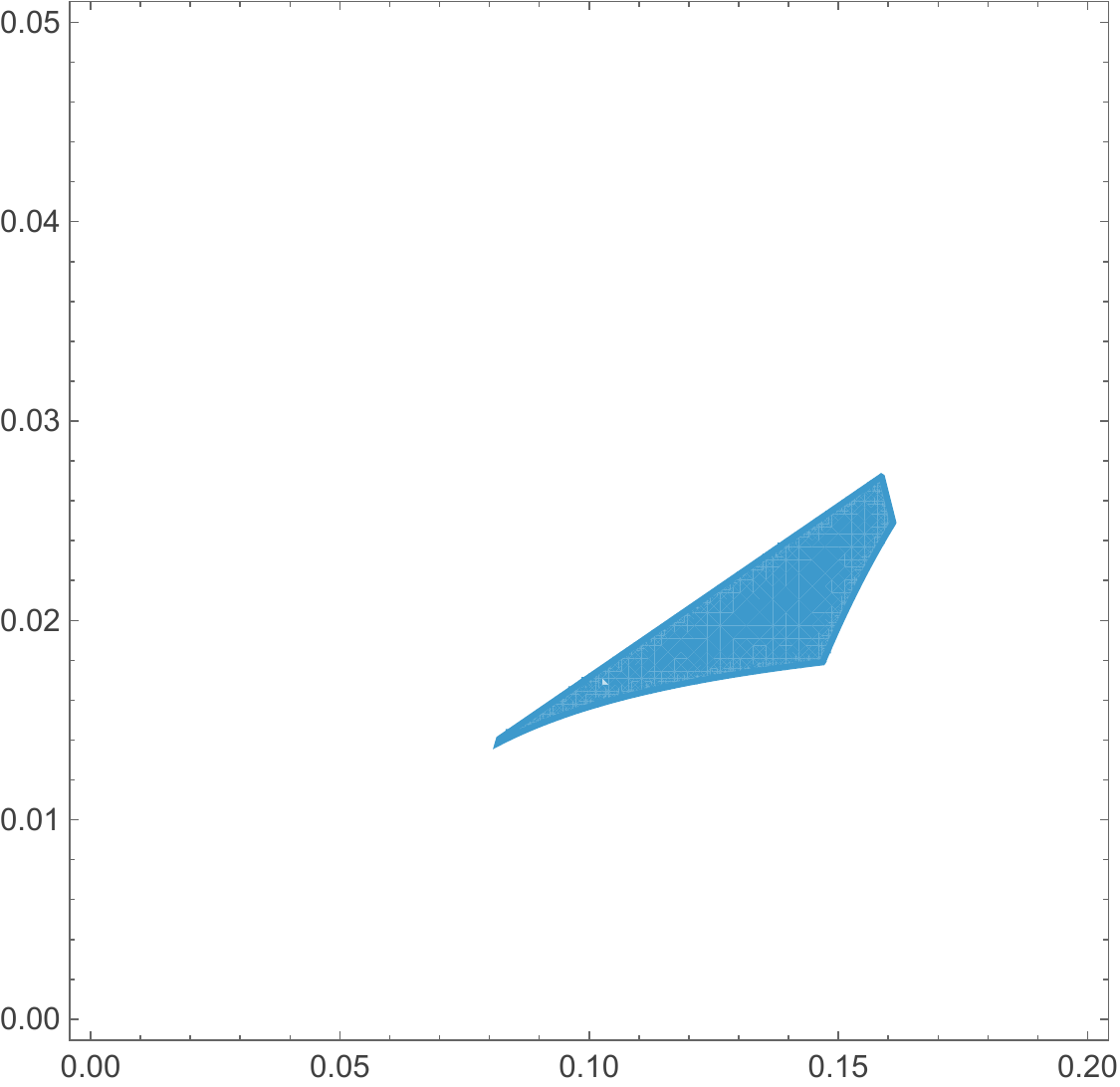}
\caption{Left: the geometric assumption on $\rch(X)$ (along the $x$-axis) and the sampling density assumption on $d_H(S,X)$ (along the $y$-axis) from Proposition \ref{Prop3Db} imply the blue region as the potential locations of $(\rch(X),d_H(X,S))$. The region is \textit{infinite} with its upper border eventually turning horizontal at height $0.0544512$. Right: the blue \textit{bounded} region as the potential locations of $(\rch(X),d_H(X,S))$ as concluded from our geometric assumption on $\rch(X)$ and the sampling density assumption, along with the computational results. Hence our PH computation imposes fairly significant restrictions on $\rch(X)$ and $d_H(S,X)$ i.e. much stronger restrictions than the mentioned assumptions themselves.  }
\label{ReachVSdH3D}
\end{center}
\end{figure}

Using proposition \ref{reach proposition}, one can show that the reach of every set contained in $\{(x,y,z)\in \mathbb{R}^3,|x|<2,|y|<5,|z|<5\}$ is bounded from below from $0.1$. Moreover, since the orbits used to define $X$ are bounded, $X$ must be compact and connected.  It has the homology groups of a torus and we can therefore state the following:
\begin{corollary}
    Let us suppose that the orbits of system (\ref{1st torus}) with initial conditions $(0.5,0,0.5),\ (0.5,0,0),$ and $(1,6,0)$ have a common $\omega$-limit set $X$. Denote by $S$ the sample point set described above. Assume that $d_H(X,S)\leq 0.017$. If $X$ is a surface , it is homeomorphic to a torus.

In particular, the system (\ref{1st torus}) possesses an invariant torus.
\end{corollary}

We have thus proved that, provided that our sample set has approached the sought-after limit set, this set is a torus and not, for example, a long periodic orbit. 

We can again draw further geometric observations based on the fact that Cech complex $\Cech(S,r)$ is homotopy equivalent to the $r$-neighborhood $S^r$ by the nerve theorem, and that $X^{r-\e}\subseteq S^r \subseteq X^{r+\e}$. 
    \begin{enumerate}
        \item As a torus, $X$ encloses a circular cave (bounded region in $\RR^3$) enclosed by the fundamental class in $H_2$:
        \begin{enumerate}
            \item Let $T_1$ denote the smallest thickenning of $X$ reaches across this cave and makes it not circular, i.e., approximately the thickness of the torus tube. We can see that $T_1 \in [0.29874988-\e,0.29874988+\e]$, where $0.29874988$ is the death time of a homology generator in $H_1(\Cech(S, r))$. We know that the other $H_1$ generator dying after $4$ does not correspond to this as it lives longer than the fundamental class in $H_2$.
            \item Let $T_2$ denote the smallest thickenning of $X$ that fills in the fundamental class. We can see that $T_2 \in [0.45340-\e,0.45340+\e]$ from the death of $H_2$. 
            \item There is another 2-dimensional bar approximately $[T_1, T_2]$. This one might appear due to the symmetry: if the torus has two symmetric thinnest parts of radius approximately $T_1$, then filling them terminates a 1-dimensional homology class and creates a 2-dim homology class by separating the cave into two caves. At $T_2$ both caves get filled in, with the coinciding death time potentially being the result of the symmetry.
            \item The second $H_1$ generator of the torus gets filled at $T_3 $, which lies on the interval $4.040292 \pm \e$.
            \item Another significant 1-dim bar indicates that $X$ is curving significantly, which we now explain.
        \end{enumerate}
                 \item \textbf{Proximity to the standard torus:} Let $Y$ be the standard torus with the tube-radius $R_1$ and inner-loop radius $R_2$. i.e., $Y$ is parameterized as
                 \[
                 \Big((R_2 + R_1 \cos(\f))\cos(\theta),
                 (R_2 + R_1 \cos(\f))\sin(\theta),
                 R_1 \sin(\f)\Big)
                 \]
                 Then the PD of $Y$ has the following bars:
         \begin{itemize}
             \item 0-dim $[0,\infty)$.
             \item 1-dim $[0,R_1], [0,R_2]$.
             \item 2-dim $[0,R_1]$.
         \end{itemize}
         They key observation here is that the death times must coincide for one 1-dim bar and one 2-dim bar. In our example, the closest such diagram would appear if the 1-dim bar $[1.506285,4.03918]$ of  length $\ell = 4.03918-1.506285$ is transformed by an $\ell/3=0.844298$- perturbation to $[2.35058, 3.19488]$ and a new 2-bar $[2.35058, 3.19488]$ arose. Thus if $Y'\subset \RR^3$ is any subset isometric to the standard torus, then the stability of persistent homology implies that $d_H(X,Y') \geq \ell/3 - \e = 0.844298-\e$. In short, $X$ is far from any standard torus. 

         \item \textbf{Proximity to a torus with mathcing radii}: Assume $Y''$ is the standard torus of the tube radius $T_1$ and the inner-radius $T_3$. $PD(Y'')$ would contain two $1$-dimensional bars: $[0,T_1]$ and $[0,T_3]$. Thus, in the optimal matching of $PD(Y'')$ and $PD(S)$, the interval $[1.506285,4.03918]$ of the latter would have to be matched to the diagonal by the stability theorem. We conclude that $d_H(X, Y'')\geq \ell/2 - \e=1.26645-\e$.
    \end{enumerate}

\begin{proposition}
\label{Prop3Dc}
Let $X$ be the $\omega$-limit set of system (\ref{1st torus}) and S the sample point set described above. Assume that $\rch(X)>0$ and $d_H(X,S)\leq \e$ satisfy the following:
\begin{itemize}
        \item Geometry assumption: $\rch(X)> \max \{ 4 d_H(X,S), 0.2987497 + d_H(X,S)\}$. 
        \item Sampling density assumption: $d_H(X,S)<  0.0995832$
    \end{itemize}
Then $X$ has the homology of a torus.
\end{proposition}

\begin{proof}
   The interval of scales $r\in [a,b]$ in Proposition \ref{Prop3Db} was chosen so that the computed PH is constant on it, and matching the homology of a torus. On the other hand, there is a larger interval $[\alpha, \alpha']$ for which  $\im \Big(H_*(\Cech(S,\alpha))\to H_*(\Cech(S,\alpha'))\Big)$ is isomorphic to the homology of a torus, even though the interval may intersect or contain other bars of the computed barcode.

    In our case we can take 
    \[
    \alpha=\max \{0.043295, 0.000514, 0.000257\}=0.043295
    \]
    as the latest birth of the targeted homology, and 
    \[
    \alpha'< \min \{0.45340, 4.040292, 0.29874988\}=0.29874988, \quad \alpha'=0.2987497
    \]
    as as slightly less than the first death of the targeted homology. By the nature of the conditions, the conclusions of (3) can be improved if we increase $\alpha$ so that $\alpha= (\alpha'-\alpha)/2$. In this case, $\alpha=\alpha'/3=0.0995832$, $\alpha'=0.2987497$, and the statement of the proposition follows directly from Corollary \ref{CorRecon3}.
\end{proof}

The domain obtained by Proposition \ref{Prop3Dc} is sketched in Figure \ref{ReachVSdH3Da}. 
    
\begin{figure}[h]
\begin{center}
\includegraphics[scale=0.3]{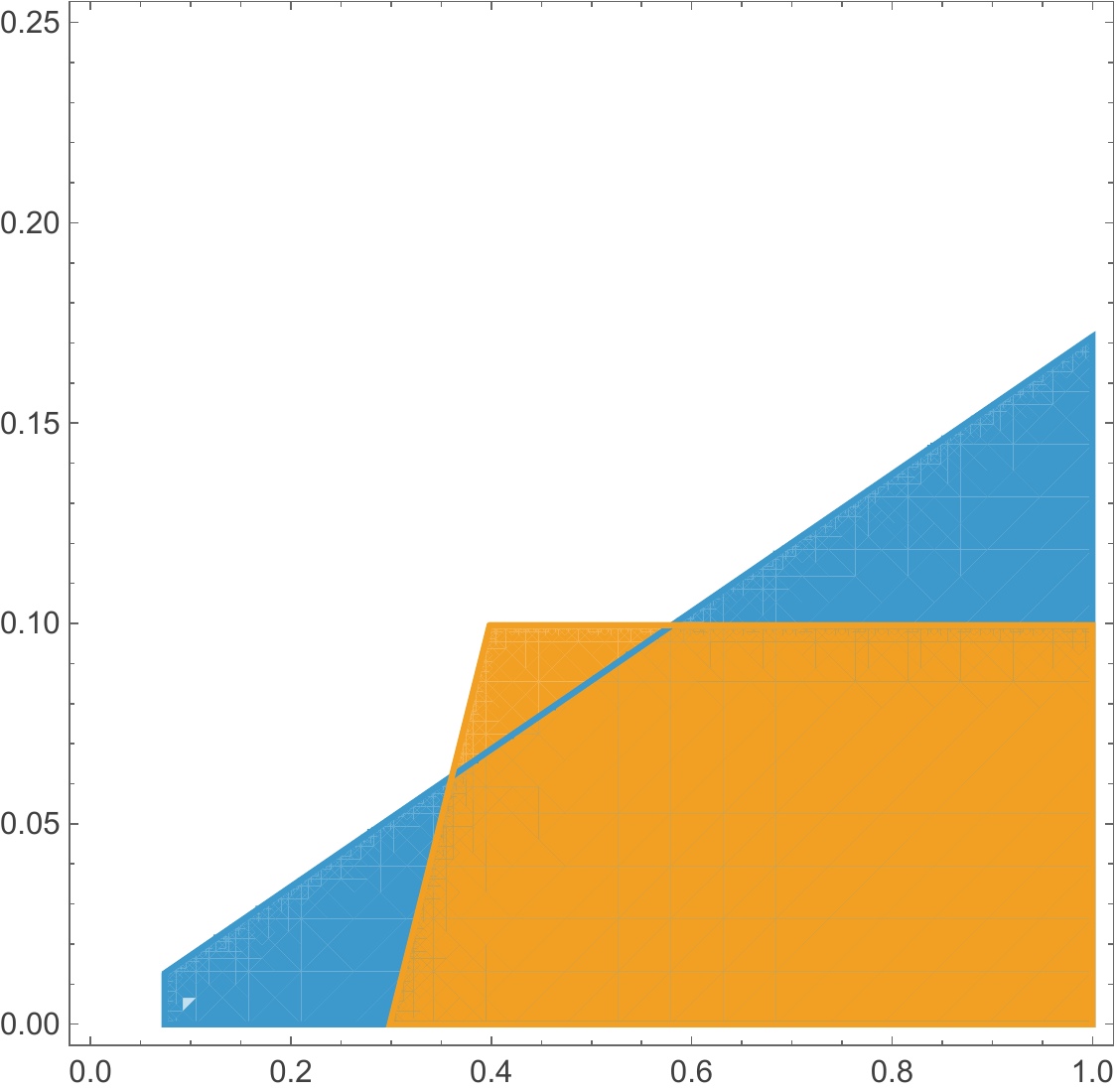}
\caption{The domain of Proposition \ref{Prop3Dc}, given by $d_H(X,S)< 0.0995832$ and $\rch(X)> \max \{ 4 d_H(X,S), 0.2987497 + d_H(X,S)\}$, is drawn in orange. It is superimposed onto the blue domain from the left side of Figure \ref{ReachVSdH3D}, which is obtained by Proposition \ref{Prop3Db}. The coordinates represent the $\rch(X)$ (along the $x$-axis) and the sampling density assumption on $d_H(S,X)$ (along the $y$-axis). We observe that Proposition \ref{Prop3Dc}, when combined with Proposition \ref{Prop3Db}, slightly extends the region of admissible configurations of $(d_H(S,X), \rch(X))$, for which the limit set attains the homology of torus. }
\label{ReachVSdH3Da}
\end{center}
\end{figure}

\subsection{Limit set for a vector field on $\mathbb{R}^4$}
In \cite{Llibre-Texeira}, the following 4d system was studied:
\begin{eqnarray}
\label{4d torus}
    \begin{cases}
    \dot{x}=-y-\epsilon(z-1)(y+2w+\frac{4x}{x^2+y^2})\\
    \dot{y}=x+\epsilon y(w+1-\frac{2}{x^2+y^2})\\
    \dot{z}=-2w\\
    \dot{w}=2z+\epsilon(y+1)w(\frac{2}{w^2+z^2}-\frac{1}{2})
    \end{cases}.
\end{eqnarray}
Set $\epsilon=0.01$ and let $(1.5,1,2,0)$ be the initial values. Numerically integrating the system (\ref{4d torus}), the projection of the orbit in the $(x,y,z)-space$ is depicted in figure \ref{fig12}.
\begin{figure}[h]
\begin{center}
\includegraphics[width=8cm,height=5cm]{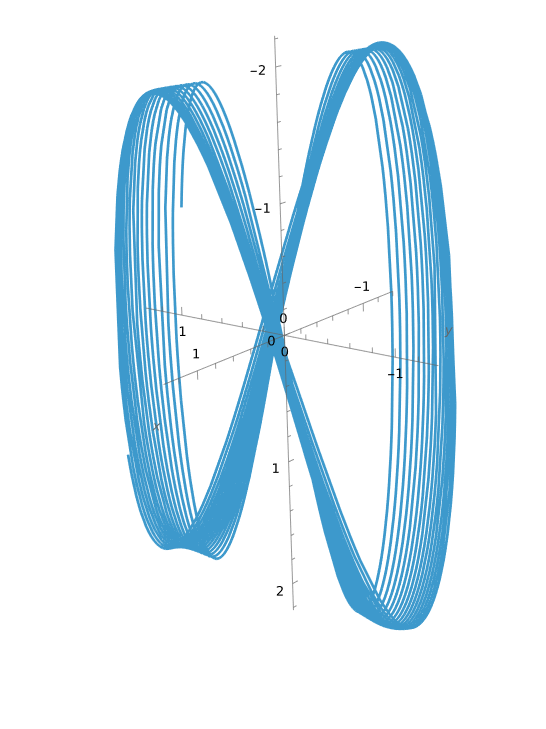}
\caption{The projection, on the $x-y-z$ space, of a numerically obtained orbit for system (\ref{4d torus}). Initial conditions: $(1.5,1,2,0)$. Here, $t\in [100,200]$.}
\label{fig12}
\end{center}
\end{figure}
Once again, to densely cover this manifold, we choose the points $(1.5,1,2,0)\ (-1,0,2,0),\ (0,1,2,0),\ (0,2,0,1)$ and numerically integrate the system, for these initial values. All four orbits seem to approach the set depicted in figure \ref{fig12}. We then sample the obtained orbits, as follows: the first one is sampled for $t\in [50,200]$ with a time-step of 0.01, the second one is sampled for $t\in[5,12]$ with a time-step of 0.001, the third one is sampled for $t\in[10,60]$ with a time-step of 0.001 and the last one is sampled for $t\in[50,430]$ with a time step of 0.01. Due to the dimensionality of the problem, the calculations were extremely time consuming. $40.000$ points were chosen for this point-set, to produce the persistence diagram shown in figure \ref{fig13}.
\begin{figure}[h]
\begin{center}
\includegraphics[width=8cm,height=7cm]{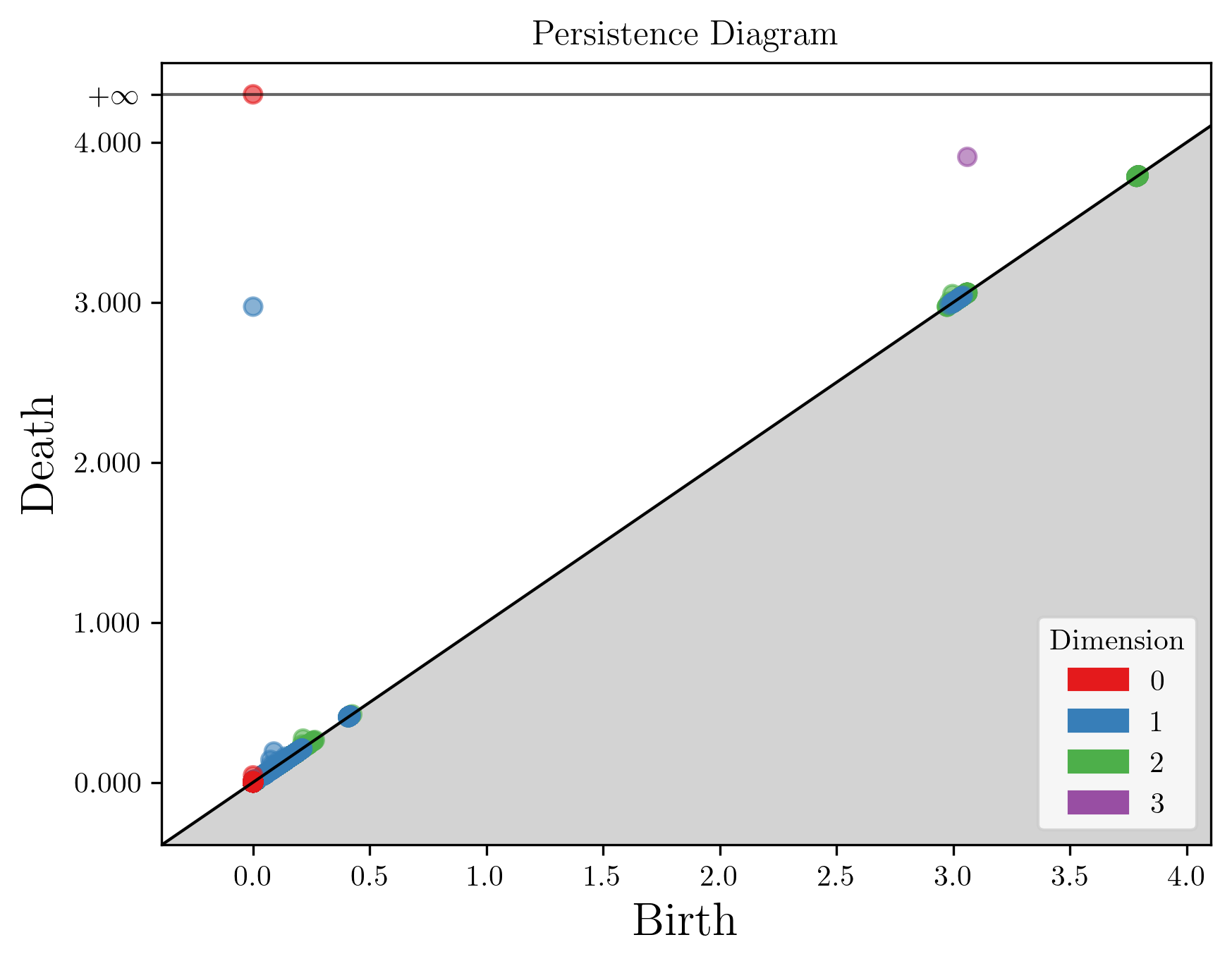}
\caption{Persistence barcode for the sample set described above.}
\label{fig13}
\end{center}
\end{figure}

As usual, let $X$ be the common limit set of the four orbits described above, which is  of positive reach, and $S$ the sample point-set of it described above. Assume $d_H(X,S)<\e$. 

We can state the first results concerning the geometry of $X$.

\begin{proposition}
\label{Prop3Da}
Let $X$ and $S$ be the $\omega$-limit set and the sample point-set described above. Assume that $\rch(X)>0$ and $d_H(X,S)\leq \e$.  
\begin{enumerate}
    \item The bottleneck distance between the geometric descriptors $PD(X)$ and $PD(S)$ is at most $\e$.
    
    \item Let $t<\infty$ be the largest finite death time of a 0-dim bar of $PD(S)$. Then $X$ cannot be represented as the disjoint union of two subspaces which are more than $2t + 2\e$ apart. 
    
    \item Let $t<\infty$ be the largest finite death time of a 0-dim bar of $PD(S)$. If $X$ is path-connected, then $d_H(S,X)\geq t$.

    \item Let $t=4.6$ be the largest distance between the points of $S$. If $\e< 4.6/2$, then the limit set can't be a point.

    \item The Hausdorff distance from $X$ to any star-like (including convex) subset of $\RR^4$ is at least  $\frac{2.97292-0.00051}{2} - \e=1.486205 - \e$.

    \item The Hausdorff distance from $X$ to any subset lying in an affine $3$-dimensional subspace is at least $\frac{3.90901-3.059879}{2} - \e=0.4245655- \e$.

\end{enumerate}
\end{proposition}

\begin{proof}
    The statements (1)-(6) are a direct consequence of Corollary \ref{CorStab}. The proof of (5) uses the fact that 
    \[(2.97292, 0.00051)\] 
    is the homology bar of the longest lifespan. The proof of (6) uses the fact that any subset lying in an affine $3$-dimensional subspace has trivial $H_3$, while $PS(S)$ has a $3$-dimensional bar $(3.90901,3.059879)$.
\end{proof}

Now, note that there is no interval $[\alpha,\alpha']$ for which 
    \[
    \im \Big(H_*(\Cech(S,\alpha))\to H_*(\Cech(S,\alpha'))\Big)
    \]
    is isomorphic to the homology of the torus.

The following proposition imposes conditions on $\rch{X}$, assuming that $X$ has the homology of a torus.

\begin{proposition}
Let $X$ and $S$ be the $\omega$-limit set and the sample point-set described above. Assume that $\rch(X)>0$ and $d_H(X,S)\leq \e$. If $X$ is a torus,  $d_H(X,S) > (3-2 \sqrt{2})\rch(X) \approx 0.171573 \rch(X)$. Even more,  $d_H(X,S)\geq \rch(X)/4$.
\end{proposition}
\begin{proof}
    The conclusion follows directly from the Theorem \ref{Recon1}, as its assumptions can not be satisfied. The final stronger inequality follows from Theorem \ref{Recon3}.
\end{proof}

\begin{remark}
    As we are thickenning $X$, at about $X^{3.0598791}$ the thickenning encloses a 3-dimensional homology class until about $X^{3.9090115}$, so there is a significant spatial curving of the torus. 
\end{remark}

However, based on the persistence diagram, the homology of $X$ may be quite different from the homology of the torus. Indeed, there are no births and deaths on the interval $(0.5, 2.9)$ and the homology on this interval is that of $S^1$. Also, there are no births and deaths on the interval $(3.06, 3.5)$ and the homology on this interval is that of $S^3$. Then Corollary \ref{CorRecon1} implies the following:

\begin{proposition}
\label{PropS1S3}
    If  $\rch(X)\geq \frac{0.5}{2-\sqrt{2}}\approx 0.853553$ and             \[
            \e \leq \min\left\{(3-2\sqrt{2})\rch(X),(3-2\sqrt{2}) \frac{2.9}{2-\sqrt{2}} \right\}\approx \min\left\{0.17157\rch(X),0.84939 \right\},
            \]
then $X$ has the homology of $S^1.$

If  $\rch(X)\geq \frac{3.06}{2-\sqrt{2}}\approx 5.22375$ and             \[
            \e \leq \min\left\{(3-2\sqrt{2})\rch(X),(3-2\sqrt{2}) \frac{3.5}{2-\sqrt{2}} \right\}\approx \min\left\{0.17157\rch(X),1.02513 \right\},
            \]
then $X$ has the homology of $S^3.$
\end{proposition}

\begin{remark}
    The assumptions of the two statements in Proposition \ref{PropS1S3} are not disjoint. Since the two conclusions are incompatible, we thus deduce that $(\rch(X), \e)$ may not satisfy both the assumptions. The situation is similar to the one represented on Figure \ref{ReachVSdH3D}, where the structure of PH enforced the restriction to the potential region of $(\rch(X), \e)$.
\end{remark}

What is more, note that $\im \Big(H_*(\Cech(S,0.0006))\to H_*(\Cech(S,2.96))\Big)$ is the homology of $S^1$. 
Then Corollary \ref{CorRecon3} for $\alpha=2.96/3,\ \alpha'=2.96$ implies the following:

\begin{proposition}
    If $\e < 2.96/3=0.986667$, and $\rch(X)> \max \{ 4 \e, 2.96 + \e\}$, then $X$ has the homology of $S^1$. 
\end{proposition}

Thus, based on the persistence diagram constructed, $X$ can be a 2-d torus, an attracting periodic orbit or even a 3-d sphere. Unfortunately, we were not able to compute a lower bound for $\rch{X}$ which would help us to exclude some of these different possibilities. The accurate computation of the reach and of the covering time of a set is a subject deserving more attention.
\section{Concluions}
In this paper we have demonstrated how to deduce the homology and homotopy type of $\omega$-limit sets of orbits of dynamical systems from the persistent homology of a sample $S$. The statements involve only the geometric assumption on the reach of the limit set and the density assumption on $S$ in terms of the Hausdorff distance. We have demonstrated our approach on four different systems. For the first three systems with the known topology of the $\omega$-limit set we obtained the right reconstruction assuming our samples are dense enough. For the fourth system with the known topology of the $\omega$-limit set we laid out the possible limits in dependence of the reach and our sample density. Furthermore, we provide a new way of estimating the reach of an invariant set of a vector field.


\end{document}